\documentclass{amsart}
\usepackage{amsmath}
\usepackage{amsfonts}
\usepackage{graphics}
\usepackage{epsfig}
\usepackage{amssymb}
\usepackage{amscd}
\usepackage[all]{xy}
\usepackage{latexsym}
\usepackage{graphicx}
\usepackage{multirow}
\usepackage{geometry}

\theoremstyle{remark}
\newtheorem{lem}{\bf Lemma}[section]

\newtheorem{thm}{\bf Theorem}[section]
\newtheorem{cor}{\bf Corollary}[section]

\newtheorem{prop}{\bf Proposition}[section]

\numberwithin{equation}{section} \numberwithin{figure}{section}

\renewcommand*{\to}{\rightarrow}
\renewcommand*{\bar}[1]{\overline{#1}}

\newcommand{\alg}{\operatorname{alg}}

\newcommand{\ch}{\operatorname{ch}}

\newcommand{\Gr}{\operatorname{Gr}}

\newcommand{\ev}{\operatorname{ev}}

\newcommand{\mb}[1]{\mathbb{#1}} 

\newcommand{\hs}{\mathcal{H}}
\newcommand{\mc}[1]{\mathcal{#1}}

\newcommand{\Spec}{\operatorname{Spec}}
\newcommand{\Proj}{\operatorname{Proj}}

\begin{document}
\large \setcounter{section}{0}

\title{Weierstrass cycles in moduli spaces and the Krichever map}

\author{Jia-Ming (Frank) Liou}
\address{Max Planck Institut f\"{u}r Mathematik\\
Vivatsgasse 7\\
Bonn, 53111, Germany\\fjmliou@gmail.com }

\author{A. Schwarz}
\address{Department of Mathematics\\
University of California\\
Davis, CA 95616, USA\\ schwarz@math.ucdavis.edu}

\begin{abstract}\large
We analyze cohomological properties of the Krichever map and use the
results to study Weierstrass cycles in moduli spaces and the
tautological ring.
\end{abstract}

\maketitle

\allowdisplaybreaks



\allowdisplaybreaks

Let us consider a point $p$ on a smooth projective connected curve $C$ over $\mb C$
of genus $g.$ We say that a natural number $n$ is a non-gap if
there exists a function that is holomorphic on $C\setminus p$ and
has a pole of order $n$ at the point $p$ (in other words
$h^0(\mathcal{O}(np))>h^{0}(\mathcal{O}((n-1)p))$).

It is obvious that  the set of all non-gaps is a
semigroup; it is easy to derive from Riemann-Roch theorem that the
number of gaps (the cardinality of the complement to the set of non-gaps in $\mathbb{N}$)
is equal to $g.$ We denote by $H$ the set consisting of $0$ and of all integers $n$ such that $h^{0}(\mc O(np))> h^{0}(\mathcal{O}((n-1)p))$ (in other words, we include $0$ and all non-gaps into $H$). One says that $H$ is the Weierstrass semigroup at $p.$

One says that a subsemigroup $H$ of $\mb N_{0}$ such that $\#(\mb N_{0}\backslash
H)=g$ and $0\in H$ is a numerical semigroup of genus $g$; obviously any
Weierstrass semigroup belongs to this class. (Here $\mb N_0$ stands for the semigroup of non-negative integers). The point $p$ is a Weierstrass point if the first non-gap is $\leq g$ (i.e. $H\neq
\{0,g+1,g+2,\cdots\}$). There exist only a finite number of
Weierstrass points on a curve. Instead of Weierstrass semigroup $H,$
one can consider a decreasing sequence of integers such that $s_{i}$
is the largest integer with
\begin{equation*}
h^{0}(K_{C}(-s_{i}p))=i.
\end{equation*}
Here $K_C$ denotes the canonical line bundle on $C$. It follows from
the Riemann-Roch theorem that this sequence (the Weierstrass
sequence of the point $p$ ) has the form $s_{i}=a_{g-i+1}-1$ if
$1\leq i\leq g$ and $s_{i}=g-1-i$ if $i\geq g+1.$ Here
$1=a_1<\cdots<a_g$ denotes the increasing sequence of gaps.

Notice that all these statements remain correct if $p$ is a
nonsingular point of an irreducible (not necessarily smooth) curve
and the canonical line bundle is replaced by the dualizing sheaf
$\omega_{C}.$ (Every irreducible curve is a Cohen-Macaulay curve;
hence it is not necessary to consider a complex of sheaves talking
about the dualizing sheaf.) {\footnote {All curves we consider are
reduced irreducible projective curves. (in other words we work with projective integral
curves)}} Any numerical semigroup of  genus $g$ is a Weierstrass semigroup at a point on an irreducible curve of  (arithmetic) genus $g$; see Section 3.

Let us consider the moduli space $\mc M_{g,1}$ of non-singular
irreducible curves of genus $g$ with one marked point (one can
characterize this space as the universal curve). \footnote {One can consider this space as an orbifold or as a moduli space of a stack. However, we are interested only in cohomology over $\mathbb{C},$ therefore it is sufficient to consider it as a topological space.}  If $H$ is a
numerical semigroup of genus $g,$ we denote by $\mc M_H$ the subset
of $\mc M_{g,1}$ consisting of curves with marked points having
Weierstrass  semigroup $H$. The closure $W_{H}=\bar {\mc M_H}$ of
the Weierstrass   set $\mc M_H$ in $\mc M_{g,1}$ is called a
Weierstrass  cycle. Under some conditions, we calculate the
cohomology class $[W_{H}]$ dual to this cycle. 
(Our methods can be used also to calculate the element of Chow ring specified by
Weierstrass cycle).

Our problem is closely related to the problem of the calculation of
the homomorphism induced by the Krichever map $k:\widehat{\mc
M}_{g}\to \Gr(\hs).$ Here $\widehat{\mc M}_{g}$ stands for the
moduli space of triples $(C,p,z)$, where $C$ is a complex
connected smooth projective curve of genus $g$ with a point $p$ and a map $z:D\to \mb
D$ is an isomorphism from a closed set $D$ onto the closed unit disk
$\mb D=\{z\in\mb C:|z|\leq 1\}$ obeying $z(p)=0$. \footnote{  The embedding into Grassmannian induces topology on  $\widehat{\mc M}_{g}$. Of course, this topology can be described without the reference to Grassmannian.} We use the
notation $\Gr(\hs)$ for the Sato Grassmannian (as defined in \cite
{SW}) and the notation $\Gr _d(\hs)$ for index $d$ component of
Grassmannian. The Krichever map sends a triple $(C,p,z)$ into the
space $V,$ the closure of functions on the boundary of the disk $D$ that can be
extended to holomorphic differentials on the complement of $D$. (A
function $f(z)$ on $S^{1}$ is considered as a differential $f(z)dz$
restricted to the boundary of $D$.) The kernel and the cokernel of
$\pi_{-}|_{V}:V\to \hs_{-}$ are identified with
$H^{0}(C,\omega_{C})$ and $H^{1}(C,\omega_{C})$ respectively (see
\cite{Mulase}, \cite{SW}); hence the index $\pi_{-}|_{V}$ is $g-1.$ 
Here $\pi_{-}|_{V}$ stands for the orthogonal projection of $V$ into $\hs_-$; the projection is defined with
respect to Hermitian inner product $\langle f_{1},f_{2}\rangle
=\int_{S^{1}}f_{1}(z)\bar{f_{2}(z)}dz/2\pi.$  Hence the image
of the Krichever map lies in the component $\Gr_{g-1}(\hs)$. The
Krichever map commutes with the natural action of $S^1$ on
$\widehat{\mc M}_{g}$ and on $ \Gr(\hs).$ Thus it induces a
homomorphism of the $S^1$-equivariant cohomology of the connected
component $\Gr_{g-1}(\hs)$ of $\Gr(\hs)$ into the $S^1$-equivariant cohomology of
$\widehat{\mc M}_{g}.$ The latter is isomorphic to the conventional
cohomology of $\mc M_{g,1}$ (see \cite {LS} for more detail). In
\cite {LS}, we have calculated the images of a set of multiplicative
generators under the homomorphism induced by the Krichever map in
the $S^1$-equivariant cohomology of Grassmannian; in the present paper, we
will give an explicit formula for this homomorphism on additive
generators of this cohomology. In the paper \cite {LS2}, we
identified the $S^1$-equivariant cohomology of Grassmannian with the ring
of shifted symmetric functions (see also \cite {LamS}). We describe
the homomorphism induced by the Krichever map on this ring; we
specify the answers for various additive generators of $S^1$-equivariant
cohomology.

Weierstrass cycles $W_{H}$ are related to intersections of Schubert
cycles in the Grassmannian with the image of Krichever map. This allows us to obtain the information about
classes $[W_{H}]$ from the analysis of the homomorphism induced by
the Krichever map in the $S^1$-equivariant cohomology. The same technique
is used to obtain relations in the tautological rings of moduli
spaces. We obtain also similar results for the moduli spaces of irreducible
(possibly singular) curves with embedded disks.

In a separate paper \cite {LSX}, we show how to use the ideas of
present paper to obtain estimates for dimensions of Weierstrass
cycles. We perform calculations for moduli spaces of irreducible
curves of low genera.

\section{Krichever Map}
In the introduction, we have described the Krichever map
$k:\widehat{\mc M}_{g}\to \Gr(\hs)$ of the moduli space
$\widehat{\mc M}_{g}$ into Segal-Wilson version of Sato Grassmannian
(see \cite {SW} for more detail). This construction can be
generalized to projective integral curves (the marked point $p$ should be
non-singular, the disk $D$ should consist of non-singular points, instead of holomorphic differentials one should
consider sections of the dualizing sheaf of $C$). This follows from the
results of \cite {SW} and from the remark that the dualizing sheaf
of Cohen-Macaulay curve is a torsion-free rank one sheaf. We will
denote by $\widehat{\mathcal{CM}}_g$ the moduli space of triples $(C,p,z)$ where $C$ is a projective integral curve of genus $g,$ and $p$ is a nonsingular point, and $z$ is a local coordinate system around $p$ sending a closed set $D$ containing $p$ onto $\mb D$ with $z(p)=0$; the extension of the Krichever map to this space will be also denoted by $k$. The extended Krichever map is an embedding of  $\widehat{\mathcal{CM}}_g$ into Grassmannian
(this follows from the results of \cite {SW}); we can define the
topology on $\widehat{\mathcal{CM}}_g$ using this embedding. The
image of this embedding is called the Krichever locus. {\footnote
{The Krichever map can be defined also for reducible curves, but in
this case this map is not  an embedding and it is not continuos. In
particular, the Krichever map on the space of nodal curves with
disks is discontinuous. }}\\

Notice that a reasonable (separable) moduli space of singular curves (even
of Gorenstein curves) does not exist;  see \cite {HM}. It is
important that we consider curves with embedded disks. Identifying
the points of  $\widehat{\mathcal{CM}}_g$ corresponding to the same
curve $C$ with different embedded disks we obtain a non-separable
space.

We have used the dualizing sheaf in the construction of Krichever
map; however, as it was shown in \cite{SW}, one can use any
torsion-free rank one sheaf. {\footnote {Notice that our definition of the Krichever locus is not quite standard. Usually this locus is defined as the image of the general Krichever map.}}

Using $q$-differentials, one can construct a more general Krichever
map $k_q: \widehat{\mc M}_{g}\to \Gr(\hs)$ for each $q\in \mathbb{Z}$; this
corresponds to using the $q$-th power of dualizing sheaf.  It is clear
that $k_1= k$. In general the map $k_q$ for $q>1$ or $q<0$ cannot be defined
for general irreducible curves, but it can be defined for Gorenstein
curves where the dualizing sheaf is an invertible sheaf (is a line bundle).

It is easy to check that the images of any triple $(C,p,z)$ under $k_q$ and $k_{1-q}$ are closed subspaces of $L^{2}(S^{1})$ orthogonal with respect to bilinear inner product
\begin{equation}\label{bln}
(f_{1},f_{2})=\frac{1}{2\pi}\int_{S^{1}}f_{1}(z)f_{2}(z)dz;
\end{equation}
in other words, we have
\begin{equation}
\label{1} k_{1-q}(C,p,z)=(k_q(C,p,z))^{\perp}
\end{equation}
where $^{\perp}$ denotes orthogonal complement (see \cite {SCH})
with respect to the bilinear inner product. In particular, for
$q=1,$
\begin{equation}
\label{2} k_{0}(C,p,z)=(k_1(C,p,z))^{\perp}.
\end{equation}
One should emphasize that (\ref {2}) is correct for all irreducible
curves, but to make sense of (\ref {1}) we should assume that $C$ is a
Gorenstein curve. All maps $k_q$ are $S^1$-equivariant; one can
study the induced homomorphisms on the $S^1$-equivariant cohomology. The
answers are formulated in terms of lambda-classes and psi-classes
(see \cite {LS} for the analysis of these problems for non-singular
curves) .

The Hodge bundle $\mb E$ on $\widehat{\mathcal{CM}}_g$ is defined as
a bundle having the space of holomorphic sections of dualizing sheaf
as a fiber. (See a rigorous definition below.) This is an $S^1$-equivariant vector bundle whose $S^1$-equivariant
Chern classes are called lambda-classes and denoted by $\lambda
_1,\cdots,\lambda_g$. Restricting them to $\widehat{\mc M}_{g},$ we
obtain conventional lambda-classes. (Recall, that the $S^1$-equivariant
cohomology of $\widehat{\mc M}_{g}$ coincides with cohomology of
universal curve ${\mc M}_{g,1}$, see \cite {LS}.) Lambda-classes can
be considered as elementary symmetric functions of lambda-roots (of
Chern roots of the Hodge bundle).

$S^1$-equivariant cohomology can be regarded as an algebra over
polynomial ring $\mathbb{C}[u],$ where $u=c_{1}(\mc O_{\mb P^{\infty}}(1)).$ The psi-class $\psi\in
H_{S^1}(\widehat{\mathcal{CM}}_{g})$ will be defined as $-u$. It was
shown in \cite {LS} that restricting to $\widehat{\mc M}_{g}$ we
obtain the standard definition of psi-class.

The subring of the ring $H_{S^1}(\widehat{\mathcal{CM}}_{g})$
generated by lambda- classes and psi-class will be called
tautological ring. It will follow from our results that the
tautological ring can be characterized as the image of $S^1$-equivariant
cohomology of Grassmannian by the homomorphism $k^*$ induced by the
Krichever map. We will prove some relations in the tautological
ring; these relations can be restricted to relations in the
tautological ring of the universal curve.

Let us consider submanifolds $\Gr_{d}^{l}$ of $\Gr_{d}(\hs)$
consisting of points $W$ such that the orthogonal projection $\pi_{l}:W\to
z^{-l}\hs_{-}$ is surjective. (Here $l\geq 0$). The action of
$S^{1}$ on $\Gr_{d}(\hs)$ generates an action on $\Gr_{d}^{l}$ for
each $l\geq 0.$ The kernels of the projection $\pi_{l}:W\to
z^{-l}\hs_{-}$ can be considered as fibers of an equivariant vector
bundle $\mathcal{E}_l$ over $\Gr_{d}^{l}$. This bundle has rank
$d+l.$ Using the Krichever map, we can embed $\widehat{\mathcal{CM}}_{g}$
into $\Gr_{g-1}^1$; the Hodge bundle is a pullback of the
$S^1$-equivariant vector bundle $\mathcal{E}_1$. (This statement can be
considered as a rigorous definition of the Hodge bundle.)

It is proved in \cite{LamS} and \cite{LS2} that the
$S^1$-equivariant cohomology ring of Grassmannian $\Gr_{d}(\hs)$ can
be identified with the ring $\Lambda^{*}(z\|u)$ of ``polynomial'' functions of variables $(z_{i})_{i\in\mb N}$ and the variable $u$ that become symmetric with respect to the variables $\{x_i\},$ where $x_{i}=z_{i}+(d+1-i)u,$ for $i\geq 1.$ These functions are called shifted symmetric functions \cite {OO}, \cite {LS2}.\\

The ring $\Lambda^{*}(z\|u)$ can be identified with the ring $\Lambda(x\|u)$ of functions $\alpha(u,x_{1},x_{2},\cdots)$ in variables $(x_{i})_{i\in\mb N}$ and in $u$  that  are symmetric with respect to the variables $(x_{i})_{i\in\mb N},$  and can be obtained from a polynomial $\widetilde{\alpha}(u,z_{1},\cdots,z_{N})\in\mb C[u,z_{1},\cdots,z_{N}],$  by means of substitution $z_{i}=x_{i}-(d+1-i)u$ for $i\geq 1.$  ( Hence $\alpha(u,x_{1},x_{2},\cdots)=\widetilde{\alpha}(u,z_{1},\cdots,z_{N}).$ The function $\alpha$ is defined on sequences $(x_{i})_{i\in\mb N}$ obeying $x_{i}=(d+1-i)u$ for $i>>0$)  


Let $k=k_{1}:\widehat{\mathcal{CM}}_{g}\to\Gr_{g-1}(\hs)$ be the Krichever map.  Let $\alpha$ be a $S^1$-equivariant cohomology class of $\Gr_{g-1}(\hs)$ represented by a function $\alpha(u,x_1,\cdots,x_i,\cdots)$ symmetric with respect to $(x_{i})$ that becomes a polynomial $\widetilde{\alpha}(u,z_{1},\cdots,z_{N})\in\mb C[u,z_{1},\cdots,z_{N}]$ with $z_{i}=x_{i}-(g-i)u$ for $i\geq 1$ and for some $N\gg 0.$
We will prove the following statements:

\begin{thm}\label{t1}
The classes $\{-k^{*}x_{i}:1\leq i\leq g\}$ are Chern roots of the bundle $\mb E^{\vee}$ dual to the Hodge bundle $\mb E$ and $k^{*}x_{i}=-(g-i)\psi$ for $i>g$  (or equivalently, $k^{*}z_{i}=0$ for all $i>g.$) It follows that  $\alpha (-\psi, k^*(x_1),\cdots, k^*(x_i),\cdots)$ is well defined.
We prove that
\begin{equation*}
k^*\alpha= \alpha (-\psi, k^*(x_1),\cdots, k^*(x_i),\cdots).
\end{equation*}
In other words
\begin{equation*}
k^*\alpha=\alpha_g(\psi,k^{*}x_1,\cdots ,k^{*}x_g).
\end{equation*}
Here we obtain $\alpha_g$ from $\widetilde{\alpha}$ by setting $\alpha_{g}(\psi,k^{*}x_{1},\cdots,k^{*}x_{g})=\widetilde{\alpha}(k^{*}u,k^{*}z_{1},\cdots,k^{*}z_{N})$ where $z_{i}=x_{i}-(g-i)u$ for $i\geq 1.$
\end{thm}
\begin{proof}
Denote $d=g-1.$ Let $\hs_{i,j}$ be the linear subspace of $\hs$ spanned by
$\{z^{s}:i\leq s\leq j\}$ and denote $\underline{\hs}_{ij}$ the
product bundle $\hs_{i,j}\times \Gr_{d}^{l}.$ We consider the
action of $S^{1}$ on $\underline{\hs}_{i,j}$ defined by
\begin{equation}\label{a}
(t,(f,V))\mapsto (t^{-1}f(t^{-1}z),t(V)).
\end{equation}
Here $V$ is a point in $\Gr_{d}(\hs)$, $f$ is vector in $\hs_{i,j}$
and $t\in S^{1}$; Here we define $t(V)$ as the space of functions
$t^{-1}f(t^{-1}z)$ for $f(z)\in V.$ Then $\underline{\hs}_{ij}$ is
a $S^1$-equivariant vector bundle over $\Gr_{d}^{l}.$ Then the total 
$S^1$-equivariant Chern classes of the bundle $\underline{\hs}_{i,j}$ is
given by the formula
\begin{equation*}
c^{T}(\underline{\hs}_{i,j})=\prod_{m=i}^{j}(1-(m+1)u).
\end{equation*}
Let $f_{ln}$ be the inclusion maps
$\Gr_{d}^{l}\hookrightarrow \Gr_{d}^{n}$ and
$\Gr_{d}^{l}\hookrightarrow \Gr_{d}(\hs)$ respectively. The induced
map of $f_{ln}$ and $f_{l}$ on the equivariant cohomology are
denoted by $f_{ln}^{*}$ and $f_{l}^{*}$ respectively. 

Let $\{x_{1},\cdots,x_{d+l}\}$ be the $S^1$-equivariant Chern roots of $\mc E_{l}^{\vee}.$
The $S^1$-equivariant cohomology $H_{S^{1}}^{*}(\Gr_{d}^{l})$ of $\Gr_{d}^{l}$ can be identified with the algebra $\Lambda(x_{1},\cdots,x_{d+l}\|u)$ of polynomials in $x_{1},\cdots,x_{d+l},u$ over $\mb C$ symmetric with respect to $\{x_{1},\cdots,x_{d+l}\}.$ The inclusion map $f_{ln}:\Gr_{d}^{l}\hookrightarrow \Gr_{d}^{n}$ induces an algebra homomorphism 
$f_{ln}^{*}:\Lambda(x_{1},\cdots,x_{d+n}\|u)\to \Lambda(x_{1},\cdots,x_{d+l}\|u).$ The projective limit  of the projective system $(\Lambda(x_{1},\cdots,x_{d+l}\|u),f_{ln}^{*})$ is the $S^1$-equivariant cohomology $H_{S^{1}}^{*}(\Gr_{d}(\hs)).$ It can be identified with the ring $\Lambda(x\|u)$ defined above.
Let $\alpha$ be an $S^1$-equivariant cohomology class in $H_{S^{1}}^{*}(\Gr_{d}(\hs))$ represented by a function $\alpha(u,x_{1},\cdots)$ in $\Lambda(x\|u).$ To compute $k^{*}\alpha,$ we only need to compute $\alpha(k^{*}u,k^{*}x_{1},\cdots).$

The Krichever map $k:\widehat{\mathcal{CM}}_{g}\to\Gr_{g-1}(\hs)$ are composition of maps:
\begin{equation*}
\begin{CD}
\widehat{\mathcal{CM}}_{g}@>\ ^{l}k>> \Gr_{g-1}^{l}@>f_{l}>>\Gr_{g-1}(\hs),
\end{CD}
\end{equation*}
where $^{l}k:\widehat{\mathcal{CM}}_{g}\to\Gr_{g-1}^{l}$ is the modified Krichever map. Notice that the pull back bundle $^{l}k^{*}\mc E_{l}$ on $\widehat{\mathcal{CM}}_{g}$ has a orthogonal direct sum decomposition:
$$^{l}k^{*}\mc E_{l}=\ ^{1}k\mc E_{1}\oplus k^{*}\underline{\hs}_{-l,-1}.$$
This implies that $\{-kx_{i}:1\leq i\leq d+l\}$ forms equivariant Chern roots of the direct sum bundles $^{1}k\mc E_{1}\oplus k^{*}\underline{\hs}_{-l,-1}.$ Hence we can set $k^{*}x_{i}=-(g-i)\psi$ for $i\geq g+1$ and $\{kx_{i}:1\leq i\leq g\}$
the equivariant Chern roots of the bundle $^{1}k^{*}\mc E_{1}.$ This proves our assertion.
\end{proof}


Schubert cycles $\bar{\Sigma}_{\mu}$ specify $S^1$-equivariant cohomology
classes $\Omega_{\mu}^{T}$ corresponding to Okounkov-Olshanski
shifted Schur functions $s_{\mu}^{*}.$ Let us recall the
definition of the shifted Schur functions following \cite{OO}. The
factorial Schur polynomial depending on partition $\mu$  and  variables $\{z_{1},\cdots,z_{n}\}$
is given by the formula:
\begin{equation*}
^{n}t_{\mu}(z_{1},\cdots,z_{n})=\frac{\det[(z_{i}\downharpoonright
\mu_{j}+n-j)]_{i,j=1}^{n}}{\det[(z_{i}\downharpoonright n-j)]_{i,j=1}^{n}},
\end{equation*}
where the symbol $(z\downharpoonright i)$ stands for the $i$-th
falling factorial power of the variable $z$:
\begin{equation*}
(z\downharpoonright i)=\left\{
                         \begin{array}{ll}
                           z(z-1)\cdots(z-i+1), & \hbox{$i=1,2,\cdots$;} \\
                           1, & \hbox{$i=0$.}
                         \end{array}
                       \right.
\end{equation*}
After the change of variables $z_{i}'=z_{i}-n+i$ for $1\leq i\leq
n,$ we obtain the shifted Schur polynomials
$^{n}s_{\mu}^{*}(z_{1}',\cdots,z_{n}')=\ ^{n}t_{\mu}(z_{1},\cdots,z_{n}).$ The
shifted Schur polynomials satisfy the stability conditions
$^{n+1}s_{\mu}^{*}(z_{1},\cdots,z_{n},0)=\ ^{n}s_{\mu}^{*}(z_{1},\cdots,z_{n})$
which allows us to define the shifted Schur functions
$s_{\mu}^{*}(z_{1},z_{2},\cdots)$ in the sequence of variables
$\{z_{1},z_{2},\cdots\}.$  The stability condition expressed in terms of factorial
Schur functions looks as follows: given $g\geq 1,$ the polynomial $^{n}t_{\mu}(z_{1}-(n-g+1),\cdots,z_{n}-(n-g+1))$ does not depend on $g$ for all $n>l(\mu).$ To be more precise,
$$^{n+1}t_{\mu}(z_{1}-(n+1-g+1),\cdots,z_{n+1}-(n+1-g+1))=\ ^{n}t_{\mu}(z_{1}-(n-g+1),\cdots,z_{n}-(n-g+1))$$
for $n>l(\mu).$ For more detains, see \cite{OO}. 
It follows from the results of \cite {LS} that the equivariant Schubert class in
$H_{S^{1}}^{*}(\Gr_{g-1}(\hs))$ corresponding to the partition $\mu$
is given by the formula:
\begin{equation*}
\Omega_{\mu}^{T}=s_{\mu}^{*}(z_{1},z_{2},\cdots)u^{|\mu|}=\ ^{n}t_{\mu}\left(\frac{x_{1}-(n-g+1)u}{u},\cdots,\frac{x_{n}-(n-g+1)u}{u}\right)u^{|\mu|},
\end{equation*}
for all $n>l(\mu).$ Here
$(z_{i})$ is the sequence of variables defined by
$z_{i}=(x_{i}+(i-g)u)/u$ for all $i,$ and $|\mu|,$ the weight of a
partition $\mu,$ is defined to be $\sum_{i}\mu_{i}.$ Note that
$x_{i}+(i-g)u=0$ for all $i$ sufficiently large in
$H_{S^{1}}^{*}(\Gr_{g-1}(\hs))$ and thus the sequence of variables
$(z_{i})$ defined by $z_{i}=(x_{i}+(i-g)u)/u$ makes sense in
$s_{\mu}^{*}$. Using this statement and the Theorem \ref{t1}, we
obtain
\begin {cor}
\label {ss}
\begin{equation*}
k^*\Omega_{\mu}^{T}=\ ^{g}s_{\mu}^{*}(z_{1},\cdots,z_{g})(-\psi)^{|\mu|}=\ ^{g}t_{\mu}\left(z_{1}',\cdots,z_{g}'\right)(-\psi)^{|\mu|}
\end{equation*}
where $\{z_{1},\cdots,z_{g}\}$ is the set of variables defined by
$z_{i}=(k^{*}x_{i}-(i-g)\psi)/(-\psi)$ for $1\leq i\leq g$ and
$\{z_{1}',\cdots,z_{g}'\}$ is the set of variables defined by
$z_{i}'=(k^{*}x_{i}+(n-g+1)\psi)/(-\psi)$ for $i\geq 1$ and $n$ is a
positive integer such that $n> l(\mu).$
\end{cor}
The factorial Schur function is an inhomogeneous symmetric function; we will
represent it as a sum of homogeneous polynomials:
$$^{n}t_{\mu}(x_{1}-(n-g+1),\cdots,x_{n}-(n-g+1))=\sum t_{\mu}^i(x_{1},\cdots,x_{n}),$$
where $t^i_{\mu}(x_{1},\cdots,x_{n})$ is a homogeneous polynomial of
degree $i$ and $n> l(\mu)$  (Recall that the LHS does not depend on $g$
for large $n$). We can write
\begin{equation}\label {tk}
k^*\Omega_{\mu}^{T}=\sum_{i}t_{\mu}^i(k^{*}x_1,\cdots,k^{*}x_g)(-\psi)^{|\mu|-i}.
\end{equation}

Shifted Schur functions form a basis in the space of all shifted
symmetric functions, and therefore we can say that conversely
Theorem \ref {t1} follows from Corollary \ref {ss}.

Denote $\Psi_{\mu}$ the $l(\mu)\times l(\mu)$ matrix whose $ij$-th
entry is given by
\begin{equation*}
(\Psi_{\mu})_{ij}=\left\{
            \begin{array}{ll}
              \sum_{a+b=\mu_{i}+j-i}h_{a}(k^{*}x_{1},\cdots,k^{*}x_{g})e_{b}(0,1,2,\cdots,\mu_{i}-i+g-1)\psi^{b}, & \hbox{if $\mu_{i}-i+g\geq 1$;} \\
              \sum_{a+b=\mu_{i}+j-i}e_{a}(k^{*}x_{1},\cdots,k^{*}x_{g})h_{b}(0,1,2,\cdots,i-\mu_{i}-g)\psi^{b}, & \hbox{if $\mu_{i}-i+g\leq 0$.}
            \end{array}
          \right.
\end{equation*}

We can also consider another matrix (of the size $l(\mu')\times
l(\mu')$)
 defined by
\begin{equation*}
(\Psi_{\mu}')_{ij}=\left\{
            \begin{array}{ll}
              \sum_{a+b=\mu_{i}+j-i}e_{a}(k^{*}x_{1},\cdots,k^{*}x_{g})h_{b}(0,1,2,\cdots,\mu_{i}'-i+g-1)\psi^{b}, & \hbox{if $\mu_{i}'-i+g\geq 1$;} \\
              \sum_{a+b=\mu_{i}+j-i}h_{a}(k^{*}x_{1},\cdots,k^{*}x_{g})e_{b}(0,1,2,\cdots,i-\mu_{i}'-g)\psi^{b}, & \hbox{if $\mu_{i}'-i+g\leq 0$.}
            \end{array}
          \right.
\end{equation*}
Here $\mu'$ denotes the conjugate partition of $\mu.$  Using the
determinant formula for double Schur functions, we obtain
\begin{equation*}
k^{*}\Omega_{\mu}^{T}=\det\Psi_{\mu}=\det\Psi_{\mu}'.
\end{equation*}
If $l(\mu)\leq g,$ $\mu_{i}-i+g>1$ for $1\leq i\leq g.$ Thus
\begin{equation*}
k^*\Omega_{\mu}^{T}=\det\left[\sum_{a+b=\mu_{i}+j-i}h_{a}(k^{*}x_{1},\cdots,k^{*}x_{g})e_{a}(1,2,\cdots,\mu_{i}-i+g-1)\psi^{b}\right]_{1\leq
i,j\leq l(\mu)}.
\end{equation*}
Similarly, if $l(\mu)\leq g$ and $\mu_{i}'-i+g>1,$ we can also obtain the dual formula
\begin{equation*}
k^*\Omega_{\mu}^{T}=\det\left[\sum_{a+b=\mu_{i}'+j-i}e_{a}(k^{*}x_{1},\cdots,k^{*}x_{g})h_{b}(1,2,\cdots,\mu_{i}'-i+g-1)\psi^{b}\right]_{1\leq
i,j\leq l(\mu')}.
\end{equation*}
These two formulas are useful when we compute the cohomology classes of the Weierstrass cycles.

We can consider also cohomology classes $p_s$ corresponding to
symmetric functions
\begin{equation*}
p_s(u,x_{1},\cdots,x_{n},\cdots)=\sum_{i=1}^{\infty}\{x_i^s-
(-1)^{s}(i-d-1)^{s}u^s\}
\end{equation*}
(these classes constitute a multiplicative system of generators of
equivariant cohomology ). Applying Theorem \ref{t1}, we obtain
\begin{cor}
\begin{equation*}
k^*p_s=\ch_{s}(\mb E)-\sum_{i=1}^{g}(i-g)^{s}\psi^{s},
\end{equation*}
where $\ch_s(\mb E)$ stands for the $s$-th component of the Chern
character of Hodge bundle $\mb E.$
\end{cor}


All statements proved above are valid not only for the space
$\widehat{\mathcal{CM}}_g,$ but also for its $S^1$-invariant
subspaces, in particular, for the subspace $\widehat{\mc M}_{g}$
consisting of smooth curves. For $\widehat{\mc M}_{g},$ some of our
statements can be simplified.

For the moduli space of pointed smooth curves \cite{Mum}, the
Mumford formula
\begin{equation}\label{Mumford}
c(\mb E)c(\mb E^{*})=1
\end{equation}
implies that $h_{a}(x_{1},\cdots,x_{g})=(-1)^{a}\lambda_{a}.$ Hence
the $\Psi$-matrix can be expressed in the form:
\begin{equation}\label{psi}
(\Psi_{\mu})_{ij}=\left\{
            \begin{array}{ll}
              \sum_{a+b=\mu_{i}+j-i}(-1)^{a}e_{b}(0,1,2,\cdots,\mu_{i}-i+g-1)\lambda_{a}\psi^{b}, & \hbox{if $\mu_{i}-i+g\geq 1$;} \\
              \sum_{a+b=\mu_{i}+j-i}(-1)^{a}h_{b}(0,1,2,\cdots,i-\mu_{i}-g)\lambda_{a}\psi^{b}, & \hbox{if $\mu_{i}-i+g\leq 0$.}
            \end{array}
          \right.
\end{equation}
If we are working with the moduli space $\widehat{\mc M}_{g},$ the
Chern character of the Hodge bundle can be expressed in terms of
kappa-classes \cite {Mum}. Therefore we obtain:
\begin {cor}
\begin{equation*}
k^*p_s=\left\{
         \begin{array}{ll}
           \sum_{i=1}^{g}(i-g)^{2r}\psi^{2r}, & \hbox{if $s=2r$;} \\
           B_{2r}\kappa_{2r}/2r-\sum_{i=1}^{g}(i-g)^{2r-1}\psi^{2r-1}, & \hbox{if $s=2r-1$.}
         \end{array}
       \right.
\end{equation*}
\end {cor}



\section {Weierstrass cycles}
The Schubert cells $\Sigma_{S}$ on $\Gr(\hs)$ are labeled by
decreasing sequences of integers $S:s_{1}>s_{2}>\cdots$ such that
the sets $S_{+}=\{s_{i}:i\geq 1\}\cap\mb Z_{+}$ and $S_{-}=\mb
Z_{-}\setminus\{s_{i}:i\geq 1\}$ are both finite sets\footnote{ Here
$\mb Z_{+}$ and $\mb Z_{-}$ are subsets consisting of nonnegative
integers and of negative integers respectively.}. The virtual
cardinality of a sequence $S$ is defined as $d=\#S_+ -\# S_-.$ The
closure of $\Sigma_{S}$ is the Schubert cycles $\bar{\Sigma}_{S}.$
Given a sequence $S,$ we define its corresponding partition $\mu$ by
$\mu_{i}=s_{i}+i-d,$ for all $i,$ where $d$ is the virtual
cardinality of $S.$ The equivariant Schubert class of
$\bar{\Sigma}_{S}$ in $H_{S^{1}}^{*}(\Gr(\hs))$ is
$\Omega_{\mu}^{T}$ where $\mu$ is the partition corresponding to
$S.$ For more details, see \cite{LS} and \cite{LS2}.

\begin{thm}\label{3.1}
A point $k(C,p,z)$ of the Krichever locus belongs to the Schubert
cell $\Sigma_{S}$ defined by the Weierstrass sequence $S$ at the
point $p$.
\end{thm}
(See \cite{AK} where this statement is attributed to Mumford.)

Assume that $H$ is a numerical semigroup of genus $g.$ Let
$A_{H}^{\alg}$ be the linear subspace of $\hs$ generated by elements
of the form $\{z^{-h}:h\in H\}$ whose closure is denoted by $A_{H}.$
Suppose that $\{h_{1},\cdots,h_{l}\}$ is a generating set of $H.$
Then $A_{H}^{\alg}=\mb C[[z^{-h_{1}},\cdots,z^{-h_{l}}]].$ The
affine curve $\Spec A_{H}^{\alg}$ is called a monomial curve. Let us
consider the filtration in $\mb C((z))$ by $\{z^{-n}\mb
C[[z]]:n\in\mb Z\}.$ There is a natural filtration of $A_{H}^{\alg}$
from the filtration of $\mb C((z)).$ Then we obtain the associated
graded algebra $\mbox{gr}(A_{H}^{\alg})$ from the filtration of
$A_{H}^{\alg}.$ The complete irreducible curve $C_{H}$ also called a
monomial curve is given by $\Proj(\mbox{gr} A_{H}^{\alg})$ and is
the one point completion of $\Spec A_{H}^{\alg}$.In other words,
$C_{H}=\Spec A_{H}^{\alg}\cup\{p\},$ where $p$ is a smooth point so
that $z(p)=0.$ We can check that $A_{H}=k_{0}(C_{H},p,z)$ and thus
$A_{H}^{\alg}$ is the space of meromorphic functions on $C$ with the
only possible pole at $p.$ Since $z^{-h}\in A_{H}^{\alg},$ we see
that $H$ is the Weierstrass semigroup at $p.$ Hence every numerical
semigroup of genus $g$ is a Weierstrass semigroup of a smooth point
on an irreducible curve of genus $g$.

The Weierstrass sequence $S$ of $(C,p,z)$ in
$\widehat{\mathcal{CM}}_{g}$ is closely related to the Weierstrass
semigroup $H$ of $(C,p)$. Let $\varsigma:\mb Z\to\mb Z$ be the
translation operator: $\varsigma(n)=n+1,$ for $n\in\mb Z.$ Then
$H=\mb Z-\varsigma(S)$ or equivalently $S=\varsigma^{-1}(\mb Z-H).$

Given a numerical semigroup $H$ of genus $g$ let $S$ be a sequence
defined by $S=\varsigma^{-1}(\mb Z-H).$ By (\ref{2}), we have
$\hs_{S}=k_{0}(C_{H},p,z)^{\perp}=k(C_{H},p,z),$ where $\hs_{S}$ is
the closed subspace of $\hs$ generated by $\{z^{s}:s\in S\}.$ Since
$\hs_{S}$ belongs to $\Sigma_{S}$ and $\hs_{S}=k(C_{H},p,z),$
$\hs_{S}$ belongs to the intersection of
$k(\widehat{\mathcal{CM}}_{g})$ and the Schubert cell $\Sigma_{S}.$
We conclude that:

\begin{thm}\label{thm3.22}
The intersection of $k(\widehat{\mathcal{CM}}_{g})$ and $\Sigma_{S}$
is nonempty if and only if the set $H=\mb Z-\varsigma(S)$ is a
numerical semigroup of genus $g.$
\end{thm}

Let us consider the closure $\bar{\Sigma}_{S}$ of a Schubert cell
$\Sigma_{S}.$ A point $k(C,p,z)$ belongs to $\bar{\Sigma}_{S}$ if
and only if the Weierstrass sequence $(s_{i}(p))$ at $p$ obeys the
relation $s_{i}(p)\geq s_{i}$ for all $i.$

\begin{lem}\label{lem1}
Let $H$ be a numerical semigroup of genus $g$ and $S=\varsigma(\mb
Z-H).$ Then $s_{i}\leq 2g-2i$ for $1\leq i\leq g$ and $s_{i}=g-i-1$
for $i\geq g+1.$
\end{lem}
\begin{proof}
This statement follows from \cite {NM}, Lemma 3.2..
\end{proof}

Let $Z$ be the sequence defined by $z_{i}=2g-2i$ for $1\leq i\leq g$
and $z_{i}=g-i-1$ for $i\geq g+1.$ Then $\mb Z-\varsigma(Z)$ is the
numerical semigroup of genus $g$ generated by $2.$ Hence $\hs_{Z}\in
k(\widehat{\mathcal{CM}_{g}})\cap \Sigma_{Z}.$ If $S$ is any
sequence so that $s_{i}\leq 2g-2i$ for $1\leq i\leq g$ and
$s_{i}\leq g-i-1,$ then $z_{i}\geq s_{i}$ for all $i$ and thus
$\hs_{Z}\in k(\widehat{\mathcal{CM}}_g)\cap\bar{\Sigma}_{S}.$

\begin{prop}
The set $k(\widehat{\mathcal{CM}}_g)\cap\bar{\Sigma}_{S}$ is
nonempty if and only if the sequence $S$ obeys $s_{i}\leq 2g-2i$ for
$1\leq i\leq g$ and $s_{i}\leq g-i-1.$
\end{prop}
\begin{proof}
We have seen that if $S$ obeys the relations, $\hs_{Z}\in
k(\widehat{\mathcal{CM}}_g)\cap\bar{\Sigma}_{S}$. Thus
$k(\widehat{\mathcal{CM}}_g)\cap\bar{\Sigma}_{S}$ is nonempty.
Conversely, assume that
$k(\widehat{\mathcal{CM}}_g)\cap\bar{\Sigma}_{S}$ is nonempty. Then
there exists a sequence $S'=(s_{i}')$ such that $s_{i}'\geq s_{i}$
and $k(\widehat{\mathcal{CM}}_{g})\cap \Sigma_{S'}\neq\phi.$ By the
theorem \ref{thm3.22}, $H'=\mb Z-\varsigma(S')$ is a numerical
semigroup of genus $g.$ By the Lemma \ref{lem1}, $s_{i}'\leq 2g-2i$
for $1\leq i\leq g$ and $s_{i}'=g-i-1$ for $i\geq g+1.$ Hence
$s_{i}\leq s_{i}'\leq 2g-2i$ for $1\leq i\leq g$ and $s_{i}\leq
s_{i}'=g-i-1$ for $i\geq g+1$ which completes the proof.
\end{proof}

Let us say that a set $\Gamma \subset X$ is a support of a cohomology class
$\xi\in H^*(X)$ if the restriction of this class to $X\setminus \Gamma$ is trivial
(i.e. $ \iota^* (\xi)=0$ where $\iota^*$ stands for the homomorphism of cohomology
groups induced by the embedding $X\setminus \Gamma \to X.$). If we consider the
$G$-equivariant cohomology of $G$-space $X,$ we can apply this notion to the
$G$-invariant subset $\Gamma.$ If $X$ is a manifold and $\Gamma$ is a  submanifold 
then $\Gamma $ is a support of the cohomology class that is dual to $\Gamma$; this cohomology class is denoted by $[\Gamma].$ This statement remains correct in the framework of equivariant cohomology and in the case when $\Gamma$ is a subvariety of a complex manifold $X.$ 
Conversely, if a submanifold $\Gamma$ having codimension $k$ is a support of a $k$-dimensional 
cohomology class $\xi$ of the manifold $X,$ then this cohomology class is proportional to the cohomology class dual to $\Gamma:$
\begin{equation}
\label{xi}
\xi=\mbox{const}[\Gamma]
\end{equation}
The same is true if $X$ is a complex manifold and $\Gamma$ is an irreducible subvariety. If $\Gamma$ is a reducible subvariety and is a support of a cohomology class having the dimension equal to the real codimension of $\Gamma,$ then this class is a linear combination of classes dual to irreducible components of $\Gamma.$ Similar statements are true in equivariant cases.

These results are well known but we were not able to find an appropriate reference.
See, however, \cite {Por} and \cite {Ful}.
Under some conditions one can apply these statements in the case when $X$ is infinite-dimensional
and $\Gamma$ has finite codimension. 

In particular, Schubert cycle $\bar{\Sigma}_{\mu}$ is a support of the $S^{1}$-equivariant cohomology
class $\Omega_{\mu}^{T}$ corresponding to Okounkov-Olshanski shifted Schur functions $s_{\mu}^{*}.$ 

Let us consider a $G$-equivariant map $f: Y\to X$ between $G$-manifolds. If an equivariant cohomology class
$\xi\in H_{G}^*(X)$ has a support that does not intersect $f(Y),$ then
$f^*(\xi)=0.$ Applying this remark to the Krichever map, we obtain the following statement.

If the intersection  of $\bar{\Sigma}_{S}$ and the
Krichever locus $k( \widehat{\mathcal{CM}}_g)$ is empty, then the
homomorphisn $k^*$ determined by the Krichever map sends the
equivariant cohomology class $\Omega_{\mu}^{T} $ into the trivial equivariant
cohomology class. Using Theorem \ref{t1}, we obtain a relation in
the tautological ring of $\widehat{\mathcal{CM}}_g$:
\begin{equation}
\label {tau}
\det\Psi_{\mu}=0.
\end{equation}
Here $\mu$ stands for a partition corresponding to the sequence $S.$
In particular, the above relation is satisfied if the sequence
violates the relations $s_{i}\leq 2g-2i$ for $1\leq i\leq g$ and
$s_{i}\leq g-i-1$ for $i\geq g+1.$ This relation can be expressed
also in terms of shifted Schur functions or factorial Schur
functions
\begin{equation}\label{relation}
s_{\mu}^{*}(z_{1},\cdots,z_{g})=t_{\mu}\left(-\frac{k^{*}x_{1}+l\psi}{\psi},\cdots,-\frac{k^{*}x_{g}+l\psi}{\psi}\right)=0,
\end{equation}
where $z_{i}=(x_{i}-(i-g)\psi)/(-\psi)$ for $1\leq i\leq g$ and
$l\geq l(\mu)-g+1.$ Probably, the most convenient way to express the
relations we found is to use the functions $t_\mu^i$ (homogeneous
components of factorial Schur functions) as in (\ref {tk}):
\begin {thm} \label {tr}
If $\mu$ is a partition corresponding to such a sequence $S$ that
one cannot find a Weierstrass sequence $S'$ obeying $S'\geq S$ then
$$\sum_{i}(-\psi)^{|\mu|-i}t_{\mu}^i(k^{*}x_1,\cdots,k^{*}x_g)=0.$$
\end {thm}
Of course, these relations are valid also in the case when we
restrict ourselves to smooth curves; we obtain relations in the
tautological ring of the universal curve $\mathcal{M}_{g,1}.$ Using
pull-push formula we get relations in $\mathcal{M}_g$:
\begin{equation} \label{w00}
\sum_{i}(-1)^{|\mu|-i}\kappa_{|\mu|-i-1}t_{\mu}^i(k^{*}x_1,\cdots,k^{*}x_g)=0.
\end{equation}
However, in the tautological rings of $\mathcal{M}_{g,1}$ and
$\mathcal{M}_g$ there exist other relations, in particular, the
relations following from the Mumford formula (\ref {Mumford}). Notice
that using (\ref {Mumford}) one can get the relations (\ref {tau})
on $\mathcal{M}_{g,1}$ with the $\Psi$-matrix  defined by
(\ref{psi}).

The theorem \ref {tr} gives an estimate of the tautological ring of
the space $\widehat{\mathcal{CM}}_g$ from above (Precise statements can be found below). 
To obtain an estimate of this ring from below, one can consider the restriction of
this ring to the fixed points of the $S^1$-action. Since the fixed
points of the $S^{1}$-action on $\Gr(\hs)$ are of the form
$\hs_{S},$ the fixed points on $\widehat{\mathcal{CM}}_{g}$
correspond to the monomial curves. For each Weierstrass sequence
$S,$ the inclusion map $\{\hs_{S}\}\to \widehat{\mathcal{CM}_{g}}$
induces a homomorphism on the equivariant cohomology:
\begin{equation*}
\ev_{S}:H_{S^{1}}^{*}(\widehat{\mathcal{CM}}_{g})\to
H_{S^{1}}^{*}(\{\hs_{S}\})\cong \mb C[\psi].
\end{equation*}
The ring homomorphism $\ev_{S}$ obeys
$\ev_{S}(\psi)=\psi$ and
$\ev_{S}(\lambda_{i})=e_{i}(s_{1}+1,\cdots,s_{g}+1)\psi^{i}$ for all
$1\leq i\leq g$. Taking the direct sum of all $\ev_{S}$ we obtain a ring
homomorphism $\ev=\bigoplus_{S}\ev_{S},$ where $S$ runs over all the
Weierstrass sequences.

The tautological ring of $\widehat{\mathcal{CM}}_g$ denoted by
$R=R(\widehat{\mathcal{CM}}_{g})$ is the $\mb Q$-subalgebra of
$H_{S^{1}}^{*}(\widehat{\mathcal{CM}}_{g})$ generated by
$\lambda_{1},\cdots,\lambda_{g}$ and $\psi.$ Consider the free
polynomial algebra $\mb Q[\Lambda_{1},\cdots,\Lambda_{g},\Psi]$
generated by commuting variables $\Lambda_{1},\cdots,\Lambda_{g},\Psi.$ The ring homomorphism $\epsilon:\mb Q[\Lambda_{1},\cdots,\Lambda_{g},\Psi]\to R$ sending $\Lambda_{i}\to\lambda_{i}$ and $\Psi\to\psi$ induces an isomorphism $\mb Q[\Lambda_{1},\cdots,\Lambda_{g},\Psi]/\ker\epsilon\cong R.$ The
tautological ring $R$ is the quotient ring of $\mb
Q[\Lambda_{1},\cdots,\Lambda_{g},\Psi]$ by the ideal of tautological
relations $I_{tau}=\ker\epsilon.$ Restricting $\ev$ to $R,$ we obtain a ring homomorphism from
$R$ to $\bigoplus_{S}\mb C[\psi].$  We obtain a homomorphism $\mb Q[\Lambda_{1},\cdots,\Lambda_{g},\Psi]\to\bigoplus_{S}\mb C[\psi]$ whose kernel is denoted by $I_{\ev}.$ It is obvious that the ideal $I_{tau}$ is contained in $I_{\ev}.$ We obtain a surjective homomorphism
\begin{equation}\label{lower}
R\to \mb Q[\Lambda_{1},\cdots,\Lambda_{g},\Psi]/I_{\ev}.
\end{equation}
Let $I$ be the ideal of $\mb Q[\Lambda_{1},\cdots,\Lambda_{g},\Psi]$
generated by $\widetilde{k}^{*}\Omega_{\mu}^{T},$ where $\widetilde{k}^{*}\Omega_{\mu}^{T}$
is the polynomial in $\mb Q[\Lambda_{1},\cdots,\Lambda_{g},\Psi]$ whose image in $R$ is $k^{*}\Omega_{\mu}^{T}$
with the property that $\Omega_{\mu}^{T}$ is the equivariant Schubert class of the Schubert
cycles $\bar{\Sigma}_{\mu}$ such that the intersection of
$\bar{\Sigma}_{\mu}$ and $k(\widehat{\mathcal{CM}}_{g})$ is empty.
We also know that $I$ is contained in $I_{tau}$ and thus we have a surjective
homomorphism
\begin{equation}\label{upper}
\mb Q[\Lambda_{1},\cdots,\Lambda_{g},\Psi]/I\to R.
\end{equation}

Recall that the Hilbert-Poincare series $P(A,t)$ of a graded algebra $A$ is the
generating function of $h_{i}(A)=\dim A_{i}$:
\begin{equation*}
P(A,t)=\sum_{i=0}^{\infty}h_{i}(A)t^{i}.
\end{equation*}

By (\ref{upper}) and (\ref{lower}), we have the following estimates
\begin{equation*}
h_{i}(A/I_{\ev})\leq h_{i}(R)\leq h_{i}(A/I).
\end{equation*}

In \cite{LSX}, we present the Hilbert-Poincare series of $\mb
Q[\Lambda_{1},\cdots,\Lambda_{g},\Psi]/I$ and of $\mb
Q[\Lambda_{1},\cdots,\Lambda_{g},\Psi]/I_{\ev}$ and estimate $h_{i}(R)$ for
curves of genus $g\leq 6.$

Every point $V$ in the Krichever locus is contained in the closed
subspace $\hs'$ of $\hs$ spanned by $\{z^{i}:i\neq -1\}.$ The space
$\hs'$ has a natural polarized structure coming from the polarized
structure of $\hs.$ This means that the Krichever map $k$ sends
$\widehat{\mc C\mc M}_{g}$ to $\Gr_{g}(\hs ').$  Schubert cells in
$\Gr_{g}(\hs ')$ are labeled by sequences $S$ obeying $s_{i}=g-1-i$
for $i\gg 0$; we will use the notation $\Sigma'_S$ for these cells.
It is easy to check that $\Sigma'_S= \Sigma_S\bigcap \Gr_{g}(\hs').$

Assume that $k^{-1}\Sigma'_{S}$ is nonempty. Then a point
$(C,p,z)\in k^{-1}\Sigma'_{S}$ if and only if $p$ has the
Weierstrass sequence $S.$ The Weierstrass cycle ${\widehat W}_S=k^{-1}\bar{\Sigma}_{S}'\bigcap 
\widehat{\mathcal{CM}}_g$ in $ \widehat{\mathcal{CM}}_g$ is a support of the cohomology class 
$k^*\Omega_{\mu}^{T},$ where
$\Omega_{\mu}^{T}$ is the equivariant cohomology class corresponding
to the Schubert cycle $\bar{\Sigma}'_{S}$ in the equivariant
cohomology of Grassmannian $\Gr(\hs ').$ Of course this statement can be applied also to the Weierstrass cycle ${\widehat W}_S\bigcap\widehat{\mathcal{M}}_g=k^{-1}\bar{\Sigma}_{S}'\bigcap 
\widehat{\mathcal{M}}_g$ in $ \widehat{\mathcal{M}}_g.$ We have mentioned that the equivariant cohomology of $ \widehat{\mathcal{M}}_g$ can be identified with cohomology of $\mc M_{g,1}$ by means of the forgetful map $\pi.$ We obtain that the Weierstrass cycle $W_S$ in $\mc M_{g,1}$ is a support of the cohomology class $(\pi ^*)^{-1}k^*\Omega_{\mu}^{T}.$ In the case when the codimension of the Weierstrass cycle $W_S$ equals to the dimension of the equivariant cohomology class $\Omega_{\mu}^{T}$ (in this case, one says that the Weierstrass cycle has expected dimension), we obtain information about the cohomology class dual to $W_S.$ (We can use the relation
between the notion of support and the notion of its dual class; see (\ref {xi}).)  Namely, if we assume that the Weierstrass cycle $W_S$ is irreducible, then the dual cohomology class $[W_S]$ is proportional to $(\pi ^*)^{-1} k^*\Omega_{\mu}^{T}.$ If we impose stronger condition that the intersection of the Krichever locus $k( \widehat{\mathcal{M}}_g)$ and the Schubert cycle $\bar{\Sigma}'_{S}$ is in general position, we can say that the coefficient of proportionality is equal to 1. If the Weierstass cycle $W_S$ is reducible, we can say that $(\pi ^*)^{-1}k^*\Omega_{\mu}^{T}$ is a linear combination of  cohomology classes dual to irreducible components of $W_S.$

Using the calculation of $k^*\Omega_{\mu}^{T} $ in
Section 2, we obtain the information about $[W_S]$ in terms of shifted Schur functions or factorial Schur functions:

\begin{thm}
If the complex codimension of $W_S$ is equal to $|\mu|=\sum \mu_i$ then
\begin{equation}
\label{w}
[W_S]=\mbox{const}\ ^{g}s_{\mu}^{*}(z_{1},\cdots,z_{g})(-\psi)^{|\mu|}=\mbox{const}\ ^{g}t_{\mu}\left(-\frac{k^{*}x_{1}}{\psi},\cdots.-\frac{k^{*}x_{g}}{\psi}\right)(-\psi)^{|\mu|},
\end{equation}
where $\mbox{const}$ is a  non-zero constant, $\mu$ is the partition corresponding to the sequence $S$ and
$z_{1},\cdots,z_{g}$ are the formal variables defined by
$z_{i}=(k^{*}x_{i}-(i-g-1)\psi)/(-\psi)$ for $1\leq i\leq
g.$
\end{thm}
To prove the theorem we notice that the complex codimension of $\bar{\Sigma}'_{S}$ in
$\Gr_g(\hs ')$ is equal to
$|S|=\sum_{i=1}^{i_{0}}(s_{i}+i-g)+\sum_{i=i_{0}+1}^{\infty}(s_{i}+i-g+1),$
where $i_{0}$ is the index so that $s_{i_{0}}\geq 0$ and
$s_{i_{0}+1}<0.$ (If $s_{i}<0$ for all $i,$ we set $i_{0}=0.$)  We associate to $S$ a partition $\mu=(\mu_{i})$ by
$\mu_{i}=s_{i}+i-g$ for $1\leq i\leq i_{0}$ and
$\mu_{i}=s_{i}+i-g+1$ if $i\geq i_{0}+1$; then the codimension is equal to $|\mu|.$
The partition corresponding to a
Weierstrass sequence has length at most $g$ by the Riemann-Roch
theorem. Therefore the factorial Schur function
$t_{\mu}(x_1-l,\cdots,x_g-l)$ is already in stable range for $l=0$.  To check that the constant in (\ref {w}) does not vanish we use Serre's theorem \cite {Ser}.

Again it is more convenient to use homogeneous components of factorial Schur functions.
Then
\begin{equation}\label {ww}
[W_S]=\mbox{const}\sum_{i}(-\psi)^{|\mu|-i}t_{\mu}^i(k^{*}x_1,\cdots,k^{*}x_g),
\end{equation}

Notice that in the case when the codimension of ${W}_S$ is
not equal to $|\mu|$
the RHS of (\ref {w}) makes sense, but is not related to $[W_S]$. One can say that it specifies the
cohomology of a``virtual'' Weierstrass  cycle. It is interesting to
notice that the multiplication rule of Schubert classes in the
equivariant cohomology of Grassmannian (see \cite {LS2},\cite {Mol})
gives a multiplication rule for ``virtual'' Weierstrass cycles.

One can consider Weierstrass cycles in $\mc M_g$ defined as images
of Weierstrass cycles in $\mc M_{g,1}$ by the forgetful map. In
other words, we define $W_{S}'$ as a subvariety consisting of curves
$C\in \mc M_g$ containing at least one point with Weierstrass
sequence $S$. Using the pull-push formula, we obtain the following
expression for the corresponding cohomology classes
\begin{equation} \label{w1}
[W_{S}']=\mbox{const}\sum_{i}(-1)^{|\mu|-i}\kappa_{|\mu|-i-1}t_{\mu}^i(k^{*}x_1,\cdots,k^{*}x_g).
\end{equation}
Here $\mu$ stands for the partition corresponding to $S$ and
$\kappa_{b}=\pi_{*}\psi^{b+1}$ are the kappa-classes. This
expression is valid if $W_{S}'$ has the expected dimension , i.e. the expression holds if the complex codimension of $W_{S}'$ in
$\mc M_{g}$ equals to $|\mu|-1.$

In a separate paper \cite {LSX}, we will apply the results of the
present paper to the moduli space of irreducible curves of low
genera. We estimate the dimension of Weierstrass cycles from below;
using the calculations of \cite {NM} and \cite {Nak}, we show that
for $g\leq 6,$ this estimate either coincides with the exact
dimension or differs by one. If our estimate coincides with the
exact dimension, we are able to calculate the homology class of a
Weierstrass cycle up to a constant factor; we performed this
calculation for $g\leq 6$. We compare the relations in the
tautological ring obtained in the present paper with the description
of the tautological ring of $\mathcal{M}_g $ obtained by Faber \cite
{Faber}.

{\bf Acknowledgements} We are indebted to  M. Movshev,  M. Mulase,
 F. Plaza-Martin and anonymous referee for useful comments. Our special thanks to B.
Osserman and V. Vologodsky for help with the generalization of our
results to non-smooth curves. Both authors thank Max-Planck
Insttiute f\"{u}r mathematik in Bonn for the generous support and
the wonderful environment. The second author was partially supported
by the NSF grant DMS-0805989.

\bibliography{sampartb}

\begin{thebibliography}{99}\large
\bibitem{AK} Arbarello, E., de Concini, C., Kac, V.G., Procesi,
C.,: Moduli Spaces of Curves and Representation Theory. Comm. Math.
Phys. 117, no. 1, 1-36 (1988).

\bibitem {OO} Okounkov, A., Olshanski, G.,: Shifted Schur Functions,
St. Petersburg Math. J. 9, 239-300 (1998).
\bibitem {Bini} Bini, G.,: Generalized Hodge Classes on the Moduli
Space of Curves. Beitr\"{a}ge Algebra Geom. 44, no.2, 559-565
(2003).
\bibitem {Faber} Faber, C.,: A conjectural description of the
tautological ring of the moduli space of curves. Moduli of curves
and abelian varieties, 109-129, Aspects Math., E33, Vieweg,
Braunschweig, (1999).
\bibitem {Ful} Fulton, W.,: Equivariant cohomology in algebraic geometry. Lecture
notes by Anderson, D (2007).
\bibitem {F} Fulton, W.,: Intersection Theory. Second edition.
Springer-Verlag, Berlin, (1998).
\bibitem {Gatto} Gatto, L., Ponza, F.,: Derivatives of Wronskians
with applications to families of special Weierstrass points. Trans.
Amer. Math. Soc. 351, no.6, 2233-2255 (1999).
\bibitem {HM} Harris, J., Morrison, I., Moduli of curves. Graduate
Texts in Mathematics, 187. Springer-Verlag, New York, 1998.
\bibitem {LamS} Lam, T., Shimozono M., $k$-double Schur functions
and equivariant (co)homology of affine Grassmannian, arXiv:
1105.2170.
\bibitem {LS} Liou, Jia-Ming., Schwarz, A.,: Moduli spaces and
Grassmannian, arXiv: 1111.1649
\bibitem {LS2} Liou, Jia-Ming, Schwarz, A.,: Equivariant Cohomology of
Infinite-Dimensional Grassmannian and Shifted Schur Functions,
arXiv: 1201.2554, to be published in Mathematical research letters.
\bibitem {Por} Portelli, D.,: On the supports for cohomology classes of complex manifolds. Rend. Istit. Mat. Univ. Trieste 44 (2012), 349-369. 

\bibitem{Loop} Pressley, A., Segal, G.,: Loop Groups, Lecture Notes in Math., 1111, Springer, Berlin, 1985.

\bibitem {LSX} Liou, Jia-Ming., Schwarz, A., Xu, Renjun, Weierstrass cycles and tautological ring for low genera
(in preparation)

\bibitem {NM} Mori, T., Nakano, T.,: On the Moduli space of pointed algebraic
curves of low genus- A computational approach. Tokyo J. Math. 27, no.
1, 239-253 (2004)
\bibitem {Nak} Nakano, T.,: On the Moduli space of pointed algebraic
curves of low genus II. Rationality. Tokyo J. Math. 31, no. 1,
147-160 (2008).
\bibitem {Mol} Molev, A. I., Sagan, B.,: A littlewood-Richardson
rule for factorial Schur functions, Trans. Amer. Math. Soc. 351, no.
11, 4429-4443 (1999).
\bibitem {Mulase} Mulase, M.,: Algebraic theory of the KP equations.
Perspective in mathematical physics, 151-217 (1994).

\bibitem {Mul} Mulase, M.,:Category of vector bundles on algebraic
curves and infinite dimensional Grassmannians, Internat. J. Math. 1,
no. 3, 293-342 (1990).

\bibitem {Mum} Mumford D.,: Towards an Enumerative Geometry of the
Moduli Space of Curves, Arithmetic and Geometry, Vol. II, 271-328,
Progr. Math., 36, Birkh\"{a}user Boston, Boston, MA, (1983).


\bibitem {Ser}  Serre, J. R.,: Alg\'ebre locale - multiplicit\'es, Lecture Notes in Mathematics, vol. 11.
Springer, Berlin Heidelberg New York 1961
\bibitem {SCH} Schwarz, A.,: Fermionic string and universal moduli space, Nucl. Phys. B317, 323 (1989)
\bibitem {SW} Segal, G., Wilson, G.,: Loop groups and equations of KdV
type. Inst. Hautes Etudes Sci. Publ Math. No. 61 (1985), 5-65.
\end {thebibliography}

\end {document}